\documentclass[12pt]{article}
\usepackage{geometry} 
\usepackage{amsmath}
\usepackage{amssymb}
\usepackage{mathabx}
\usepackage{amscd}
\usepackage{latexsym}
\usepackage{amsthm}
\usepackage{mathrsfs}
\usepackage{color}
\usepackage{amsfonts}
\usepackage{array}
\usepackage{setspace}
\usepackage{algorithm}
\usepackage{algpseudocode}
\usepackage{mathtools}
\usepackage{tikz}
\usepackage{tikz-cd}
\usepackage{hyperref}
\usetikzlibrary{arrows}
\usetikzlibrary{knots}
\usepackage{ytableau}
\usepackage{stmaryrd,graphicx}
\usepackage{cancel}
\usepackage{algpseudocode}
\usepackage[shortlabels]{enumitem}
\newtheorem{thma*}{Theorem A}
\newtheorem{thmb*}{Theorem B}
\newtheorem{thm}{Theorem}[section]
\newtheorem*{thm*}{Theorem}
\newtheorem{cor}[thm]{Corollary}
\newtheorem*{cor*}{Corollary}
\newtheorem{lem}[thm]{Lemma}
\newtheorem*{lem*}{Lemma}
\newtheorem{prop}[thm]{Proposition}
\newtheorem*{prop*}{Proposition}

\newtheorem*{thma}{Theorem A}
\newtheorem*{thmb}{Theorem B}

\theoremstyle{definition}
\newtheorem{defn}[thm]{Definition}
\newtheorem*{defn*}{Definition}

\newtheorem*{conjecture*}{Conjecture}

\newtheorem*{condition*}{Condition}
\newtheorem*{assumption*}{Assumption}
\newtheorem{algo}{Algorithm}
\newtheorem*{algoa}{Algorithm 1$'$}

\theoremstyle{remark}
\newtheorem{rem}[thm]{Remark}
\newtheorem*{rem*}{Remark}
\newtheorem{example}[thm]{Example}

\newtheorem*{problem*}{Problem}

\DeclareMathOperator{\readword}{rw}
\DeclareMathOperator{\lm}{LM}
\DeclareMathOperator{\sort}{sort}
\DeclareMathOperator{\osp}{OSP}
\DeclareMathOperator{\shuff}{Sh}

\DeclareMathOperator{\maj}{maj}
\DeclareMathOperator{\inv}{inv}
\DeclareMathOperator{\Des}{Des}

\DeclareMathOperator{\Inv}{Inv}

\newcommand{\antisym}{N}
\newcommand{\flag}{\mathcal{F}}
\DeclareMathOperator{\springer}{Sp}
\newcommand{\ribtabs}{\mathcal{R}}

\newcommand{\jmaj}{J^{\maj}}
\newcommand{\invtabs}{\mathcal{A}}
\newcommand{\majtabs}{\mathcal{D}}
\newcommand{\desc}{\majtabs}

\newcommand{\Q}{\mathbb{Q}}
\newcommand{\C}{\mathbb{C}}

\newcommand{\jlamaj}{J^{\maj}_{\lambda}}

\newcommand{\gpring}{R}
\newcommand{\ghideal}{\mathcal{I}}

\newcommand{\rib}{\mathcal{R}}

\newcommand{\xn}{{\mathbf{x}}}

\newcommand{\an}{{\mathbf{a}}}

\newcommand{\bn}{{\mathbf{b}}}
\newcommand{\Ct}{\tilde{C}}
\newcommand{\Ht}{\tilde{H}}

\newcommand{\hs}{\mathcal{H}}
\newcommand{\frob}{\mathcal{F}}
\newcommand{\raybas}{\mathcal{B}^{\maj}}

\newcommand{\revla}{\overleftarrow{\lambda}}

\newcommand{\PF}{\mathcal{PF}}

\newcommand{\basis}{\mathcal{B}}

\newcommand{\dinvn}{\mathbf{dinv}}

\DeclareMathOperator{\dyckpath}{path}

\DeclareMathOperator{\artin}{\mathcal{A}}

\DeclareMathOperator{\invt}{invt}
\DeclareMathOperator{\majt}{majt}
\DeclareMathOperator{\dinv}{dinv}
\DeclareMathOperator{\Dinv}{Dinv}
\DeclareMathOperator{\area}{area}

\DeclareMathOperator{\doff}{doff}

\DeclareMathOperator{\height}{ht}
\DeclareMathOperator{\lev}{lev}

\title{A Descent basis for the Garsia-Procesi module}
\author{Erik Carlsson, Raymond Chou}

\begin{document}

\maketitle

\begin{abstract}

We assign to each Young diagram $\lambda$ a subset $\raybas_{\lambda'}$ 
of the collection of Garsia-Stanton descent monomials, 
and prove that it determines a 
basis of the Garsia-Procesi module $R_\lambda$, 
whose graded character is the Hall-Littlewood polynomial $\tilde{H}_{\lambda}[X;t]$
\cite{garsia1992certain,concini1981symmetric,tanisaki1982defining}.
This basis is a major index analogue of the basis 
$\mathcal{B}_\lambda \subset R_\lambda$ defined by certain recursions
in \cite{garsia1992certain},
in the same way that the descent basis of  
\cite{garsia1984group} is related to the 
Artin basis of the coinvariant algebra $R_n$, which in fact corresponds to the 
case when $\lambda=1^n$.
By anti-symmetrizing a subset of this basis with respect to the corresponding Young subgroup under the Springer action, 
we obtain a basis in the parabolic case, as well as a corresponding formula for the expansion of $\tilde{H}_{\lambda}[X;t]$.
Despite a similar appearance, 
it does not appear obvious how 
to connect these formulas appear to the specialization of
the modified Macdonald formula of
\cite{haglund2005combinatorial} at
$q=0$.

\end{abstract}

\section{Introduction}

The coinvariant algebra is the quotient ring
\[R_n=\Q[\xn]/I_n,\quad
I_n=(e_1(\xn),...,e_n(\xn))\]
where $e_j(\xn)$ is the elementary symmetric
polynomial in the set of variables $\xn=\{x_1,...,x_n\}$.
It is isomorphic to the cohomology ring of the usual complex flag variety $R_n \cong H^*(\mathcal{F}_n)$,
where the $x_i$ correspond to
the Chern classes of the quotient
line bundles, so that its graded dimension is the $q$-factorial $[n]_q!$. 
There is an action of the symmetric group, which is a simple instance of the Springer action, defined by $\sigma \cdot f(x_1,...,x_n)=
f(x_{\sigma_1},...,x_{\sigma_n})$,
whose graded Frobenius character
$\frob_q R_n$ 
is the Hall-Littlewood polynomial $\Ht_{1^n}(X;q)$ associated to a vertical strip.

In addition to the basis of Schubert polynomials, 
there are two well-known monomial bases.
The first, due to Artin \cite{artin1944galois}, consists of the sub-staircase monomials
\begin{equation}
\label{eq:artinbasis}
\basis_n=\left\{x_1^{a_1}\cdots x_n^{a_n}: 0 \leq a_k < k\right\}
\end{equation}
whose graded sum is clearly the $q$-factorial.
These also turn out to correspond to leading terms with respect to the Gr\"{o}bner basis of $I_n$ determined by the lex order.
A second description, in which the basis is indexed by permutations, is \[\basis_n=\{f_\sigma(\xn):\sigma\in S_n\},\quad
f_{\sigma}(\xn)=\prod_{i<j:\sigma_i>\sigma_j} x_{\sigma_i}\]
so that taking degrees gives the inversion formula
\[\sum_{\sigma \in S_n} q^{\inv(\sigma)}=[n]_q!,\quad
\inv(\sigma)=\#\{(i,j):1\leq i<j\leq n,\,\sigma_i>\sigma_j\}.\]

The second basis is known as the Garsia-Stanton descent basis \cite{garsia1984group}, given by
\[\basis'_{n}=\{g_\sigma(\xn):\sigma\in S_n\},\quad
g_\sigma(\xn)=\prod_{i:\sigma_{i}>\sigma_{i+1}} x_{\sigma_1}\cdots x_{\sigma_i}.\]
Writing $\xn^{\an}=x_1^{a_1}\cdots x_n^{a_n}$ for a composition
$\an=(a_1,...,a_n)$, the set of all possible exponents, called descent compositions, is denoted by
\[\desc_n=\left\{\an: \xn^{\an}=g_\sigma(\xn)\mbox{ for some $\sigma\in S_n$}\right\}\]
The function that assigns $\sigma$ to the corresponding composition 
is called the major index table $\an=\majt(\sigma)$, whose norm is the major index.
Counting those elements with degree gives
MacMahon's formula \cite{macmahon1913indices}:
\[\sum_{\sigma \in S_n} t^{\maj(\sigma)}=[n]_t!,\quad
\maj(\sigma)=\sum_{1\leq i\leq n-1:\sigma_{i}>\sigma_{i+1}} i.\]

The monomials $g_\sigma(\xn)$ may also be described as leading terms, but for a different ordering of the monomials called the descent order,
determined by $\xn^{\an}\leq_{des} \xn^{\bn}$ if
\begin{enumerate}
    \item $\sort(\an,>) <_{lex} \sort(\bn,>)$ or
\item $\sort(\an)=\sort(\bn)$ and
$\an \leq_{lex} \bn$
\end{enumerate}
Note that $\leq_{des}$ is not a monomial order for the purposes of Gr\"{o}bner bases. E. Allen gave an explicit algorithm for reducing any monomial in this order \cite{AllenDescent}.

Given a Young diagram $\lambda=(\lambda_1,...,\lambda_l)$ of size $n$, let $\springer_\lambda \subset \flag_n$ denote the Springer fiber associated to a nilpotent matrix with Jordan blocks of given by the elements of $\lambda$. It was shown in \cite{springer1978construction} that there is an action of the symmetric group on the cohomology
ring $H^*(\springer_\lambda)$ which
is compatible with the action on $R_n$
under the pullback of the inclusion
$i_\lambda : \springer_\lambda\rightarrow \flag_n$, and realizes a graded version of the induced representation from the sign representation of the Young subgroup
$S_\lambda=S_{\lambda_1}\times \cdots \times S_{\lambda_l}\subset S_n$.
It was shown in \cite{springer1976trigonometric} that the graded decomposition of this action realizes a version of the Hall-Littlewood polynomials.
In \cite{concini1981symmetric}, De Concini and Procesi gave an explicit description of $H^*(\springer_\lambda)$ as a ring.
Tanisaki \cite{tanisaki1982defining}, 
then gave an explicit set of generators 
of the ideal $I_\lambda$ which is the kernel of the map $\Q[\xn]\rightarrow H^*(\springer_\lambda)$ in terms
of certain elementary symmetric functions
in any $k$-element subset of the variables $\xn$:
\[I_\lambda=\left\{e_r(x_{i_1},...,x_{i_k}): \{i_1,...,i_k\}\subset \{1,...,n\},\ k\geq r > k-d_k(\lambda) \right\},\]
noticing that $I_{\lambda}$ specializes to $I_n$ when $\lambda=(1^n)$.
Here $d_k(\lambda)$ are certain numbers bounding the possible values of $r$
defined below. 
Garsia and Procesi \cite{garsia1992certain}
studied the graded $S_n$-module structure of 
the quotient ring denoted by
$R_\lambda=\Q[\xn]/I_\lambda$. They analyzed some 
recursively defined
as monomials in the generating variables
$x_1,...,x_n$. In the case of a vertical strip $\lambda=(1^n)$ this
basis specializes to the Artin basis $\{f_\sigma(\xn)\}$.

In this paper, we will study monomial basis of $R_\lambda$
which specializes to the descent monomials in the same way that the
Springer monomials specialized to the Artin basis. It turns out to have a simple description as $\majtabs_{\lambda'}$ where
\[\desc_{\lambda}=\left\{\an: 
\an|_{A_i}\in \desc_{\lambda_i} \mbox{ for some $(A_1|\cdots|A_l)$}\right\}\]
where $(A_1|\cdots |A_l)$ is an ordered set partition of
$\{1,...,n\}$ with sizes $|A_i|=\lambda_i$, and $\an|_{A_i}$ denotes the composition of length $\lambda_i$ consting of those $a_j$ for $j\in A_i$, written in the original order.
In fact there may be more than one way to represent the elements of $\desc_{\lambda}$, and as a result, it turns out to be surprisingly difficult to show that this set has the correct size to be a basis of $R_{\lambda'}$, which is multinomial coefficient in $\lambda'$.
It turns out that $\desc_{\lambda} \subset \desc_n$, and thus we may define the corresponding set of permutations
\[J_{\lambda}^{\maj}=\left\{\majt^{-1} (\an): \an \in \desc_{\lambda}\right\}.\]

We then prove two theorems. The first shows that the corresponding monomials, which we denote by $\raybas_{\lambda}$, constitute a basis:
\begin{thma}
The Garsia-Stanton monomials 
$g_\sigma(\xn)$ for $\sigma \in \jmaj_{\lambda'}$ are a vector space basis of the Garsia-Procesi module $R_{\lambda}$. Moreover, they are the leading terms for the descent order, 
and we obtain the expected statements for the parabolic versions by intersecting $\jmaj_{\lambda'}$ with 
where $\shuff'(\mu)\subset S_n$, the set of
maximal left coset representatives (called reverse shuffles) of 
$S_\mu\backslash S_n$ in the Bruhat order.
\end{thma}
The proof begins by defining a map from $R_{\lambda'}$ into a certain direct sum of tensor products of usual coinvariant algebras. 
We then use properties of the descent order to show that the image of $\raybas_{\lambda}$ is linearly independent, so that $\raybas_{\lambda}$ must be independent as well. We complete the proof by showing that $\raybas_{\lambda}$ has the desired size, by identifying it with a certain set of parking functions. We then confirm the size of that set by invoking one form of the compositional shuffle theorem \cite{haglund2005combinatoriala,carlsson2018proof}.
\begin{thmb}
The modified Hall-Littlewood polynomial is given by $\Ht_{\lambda}[X;t]=\Ct_{\lambda}[X;t]$, where
\begin{equation}
\label{eq:thmb}
    \Ct_{\lambda}[X;t]=
\sum_{\mu} \Big(
\sum_{\sigma \in \jmaj_\lambda \cap
\shuff'(\mu)}
t^{\maj(\sigma)} \Big) m_{\mu}(X).
\end{equation}
\end{thmb}

\subsection{Relationship with the 
full Macdonald formula}

The motivation for this paper had to do with extending the methods of Oblomkov and the first author \cite{carlsson2018affine}
to search for a monomial basis for the Garsia-Haiman module, a long standing open problem \cite{garsia1999explicit,assaf2009kicking}. We hope this will also lead to connections with Springer varieties 
and their affine analogs.

In this direction, it is natural to ask if \eqref{eq:thmb}
can be realized as the
special case of the full Macdonald formula \cite{haglund2005combinatorial}
corresponding to setting $q=0$,
which was specifically addressed
in that paper. We were 
not able to find a reasonable bijection between the corresponding diagrams
in the two formulas.
Our closest version involved a complicated sequence involving the 
version of the RSK algorithm given
in \cite{haglund2005combinatorial},
followed by switching the resulting
two tableaux, followed by an application
of the usual RSK algorithm. Instead,
we used a different specialization
arising from taking lowest degree terms in $q$ from the Compositional Shuffle Theorem, given in terms of certain Ribbon-shaped tableaux in Proposition 
\ref{prop:ribbonhall}. Thus, any 
formula recovering \eqref{eq:thmb}
would amount to a bijection between 
those two known specializations of previously known formulas. 
We expect that finding such a bijection may also be useful in the effort to extend this basis to the full Garsia-Haiman modules.

\subsection{Acknowledgments}

Both authors were supported by NSF DMS-1802371 during this
this project, which they gratefully acknowledge. We thank A. Mellit for valuable conversations, 
particularly for suggesting to look at the ``$\dinv+\doff$''
version of the 
Compositional Shuffle Theorem.

\section{Preliminaries}

We review our preliminary notations and background on parking function statistics, symmetric functions,
and the Garsia-Procesi modules.

\subsection{Combinatorial notations}

Let $\lambda=(\lambda_1,...,\lambda_l)$ denote a partition of $n$ of length $l$,
written $\lambda \vdash n$ and $l(\lambda)=l$. 
Its Young diagram will be written in English notation, meaning boxes the rows are drawn top to bottom.
For instance, $\lambda=(3,3,2,1)$
would be represented by

\begin{center}
\begin{ytableau}
$\ $ & & \\
  & & \\
   & \\
 \\
\end{ytableau}
\end{center}
We will denote the length by $l=l(\lambda)$, and use $\lambda'$ to denote the conjugate partition. The length of the conjugate partition 
$l(\lambda')$ will be denoted $h=\lambda_1$. We will make use of the statistics from Macdonald's book given by
\[n(\lambda)=\sum_{i=1}^h
\binom{\lambda'_i}{2}.\]
More generally, we will also use the letters $\lambda,\mu,\nu$
to denote strong compositions
of $n$, which are not necessarily sorted.

An ordered set partition is a collection of $l$ subsets written $(A_1|\cdots |A_l)$ which are a partition of 
$[n]=\{1,...,n\}$. The sizes of the sets
$A_i$ determine a composition, which is called the type of the partition.
We denote the set of all ordered set partitions
of type $\lambda$ will be
denoted by $\osp(\lambda)$.

Let $S_n$ denote the symmetric group on $n$ elements. If $\mu=(\mu_1,...,\mu_l)$
is a composition, we have the Young subgroup
\[S_\mu=S_{\mu_1}\times \cdots \times S_{\mu_l} \subset S_n.\]
We have both the left and right cosets
$S_\mu \backslash S_n$ and $S_n/S_\mu$,
and both cosets have unique minimal and maximal representatives in the Bruhat order. A minimal (resp. maximal) representative of $S_\mu\backslash S_n$ is called a $\mu$-shuffle (resp. reverse shuffle) denoted $\shuff(\mu)$ and $\shuff'(\mu)$ respectively.
The elements of $\shuff(\mu)$ are permutations with the property that the elements of $A_i$ appear in sorted order.
for instance,
\[\shuff((3,2))=
\left\{ (1,2,3,4,5),(1,2,4,3,5),(1,2,4,5,3),(1,4,2,3,5),(1,4,2,5,3),\right.\]
\[\left.(1,4,5,2,3),(4,1,2,3,5),(4,1,2,5,3),(4,1,5,2,3),(4,5,1,2,3) \right\} \]
whereas $\shuff'((3,2))$ would have the elements $\{1,2,3\}$ and $\{4,5\}$ appear in decreasing order. The elements of $
\shuff(\mu),\shuff'(\mu)$ are in bijection with the set of ordered set partitions of type $\mu$.
See \cite{haglund2008catalan} for more details.

\subsection{Symmetric functions}

We will use capital letter 
$X$ to denote an alphabet in infinitely many variables $x_1,x_2,...$.
Given a field $F$, let $\Lambda_F=\Lambda_F(X)$ be the graded
algebra of symmetric polynomials in $X$
over $F$. We will denote $\Lambda=\Lambda_\Q$, and also
$\Lambda_{t}=\Lambda_{\Q(t)}$ when the subscript is a variable, and similarly for several variables. 
A superscript $\Lambda^{(n)}$ will represent those symmetric functions of degree $n$.
We have the usual bases
$s_\mu,e_\mu,h_\mu,m_\mu,p_\mu$ of 
$\Lambda^{(n)}$ consisting of the
Schur, elementary, complete, monomial,
and power sum
symmetric functions as
$\mu$ ranges over partitions of $n$,
see \cite{macdonald1995symmetric}.

If $f\in \Lambda$ is a symmetric function,
and $A$ is an expression in some variables
the plethystic substitution $f[A]$,
see \cite{Haiman01vanishingtheorems}.
It is the image of $f$ under the homomorphism defined on the power sum generators by $p_k \mapsto A^{(k)}$,
where $A^{(k)}$ is the result of evaluating $x=x^k$ for every variable $x$ appearing in $A$. This applies to 
symmetric function variables $x_i$ as 
well, and in fact we may write
\[f[x_1+x_2+\cdots]=f(x_1,x_2,...,).\]
For this reason, we may also write
$f[X]$ using $X$ to denote the set of variables, as well as the sum $x_1+x_2+\cdots$. In this way we can obtain expressions such as
\[h_2[X(1-q)^{-1}]=
\frac{1}{(1-q)(1-q^2)} m_{2}+\frac{1}{(1-q)^2} m_{1,1}.\]

If $V$ is a representation of $S_n$, we have its Frobenius character denoted
\begin{equation}
    \label{eq:frobdef}
\frob (V)=\sum_{\mu \vdash n}    
a_{\mu} s_\mu=
\sum_{\mu} b_\mu m_\mu.
\end{equation}
where $a_\mu$ is the multiplicity of the irreducible representation $\chi_\mu$ in $V$. The coefficients $b_\mu$ represent the dimension of the invariants of $V$ under the Young subgroup 
$b_\mu= \dim(V^{S_\mu})$.
If $V$ is graded, we also have the graded Frobenius character
\[\frob_{q}(V)=\sum_{i\geq 0} q^i 
\frob(V^{(i)})\]
and similarly with more variables and multiple gradings. We have other relationships with the plethystic operations. For instance, $\omega \frob(V)$
is the Frobenius charcter of the twist of
$V$ with the sign representation.

Given $\lambda \vdash n$,
the modified Macdonald polynomials
are defined by 
\[\Ht_\lambda[X;q,t]=t^{n(\lambda)}
J_\lambda[X(1-t^{-1})^{-1};q,t^{-1}]\]
where $J_\lambda[X;q,t]$ is the integral form of the Macdonald polynomial.
Haiman showed that they are the Frobenius character of the $n!$ dimensional
Garsia-Haiman module \cite{haiman2001hilbert}, and so have an
expansion which is positive in the Schur
basis. Moreover, the constant term in the monomial basis is a polynomial in $q,t$
with coefficients summing to $n!$.
The modified Hall-Littlewood polynomials 
are the specialization of the modified Macdonald polynomials at $q=0$, which
will be written $\Ht_\lambda[X;t]=
\Ht[X;q,t]$. For instance,

It follows from the axioms defining $J_\lambda[X;q,t]$ that the Hall-Littlewood polynomials are uniquely determined by
\begin{equation} 
\label{eq:HL}
\begin{split}
(T1) & \ \  \omega \Ht_{\lambda} [X;t] = \sum_{\mu \trianglerighteq \lambda'} a_{\lambda',\mu}(t)m_{\mu} \\
(T2) & \ \ \Ht_{\lambda} [(t-1)X;t] =
\sum_{\mu \trianglelefteq \lambda} b_{\lambda,\mu}(t)m_{\mu}\\
(T3) & \ \
a_{\lambda',\lambda}(t)=t^{n(\lambda)}
\end{split}
\end{equation}
For instance,
$\Ht_{3,1}[X;t]=ts_{3,1}+s_{4}$ so that
\[\omega \Ht[X;t] =tm_{{2,1,1}}+ \left( 3t+1 \right) m_{{1,1,1,1}}\]
and
\[\omega \Ht[X(1-t);t] =
{t}^{3} \left( -1+t \right) ^{2}m_{{3,1}}+{t}^{2} \left( -1+t \right) ^
{3}m_{{2,2}}+\]
\[2{t}^{2} \left( -1+t \right) ^{3}m_{{2,1,1}}+ \left( 3
t+1 \right)  \left( -1+t \right) ^{4}m_{{1,1,1,1}}.\]

\subsection{Permutation statistics}

For more on this section we refer to \cite{haglund2008catalan,stanley2011enumerative}.

For each $\sigma \in S_n$, we have the
inversion and descent set
\[\Inv(\sigma)=\left\{(\sigma_i,\sigma_j):i<j,\ \sigma_i>\sigma_j\right\},\quad \Des(\sigma)=\left\{1\leq i\leq n-1 : \sigma_i>\sigma_{i+1}\right\}.\]
For instance, for $\sigma=(3,5,1,2,4)$ we have
\[\Inv(\sigma)=\left\{(3,1),(3,2),(5,1)(5,2),(5,4)\right\},\quad \Des(\sigma)=\{2\}.\]
Note our definition for Inv differs from the literature by recording $(\sigma_i,\sigma_j)$ in place of $(i,j)$. This is done so to drastically simplify notation. An entry $i$ is called an \emph{inverse descent} of $\sigma$ if $i+1$ appears to the left of $i$ in one-line notation, or equivalently if
$i$ is a descent of $\sigma^{-1}$. We define $\text{iDes}(\sigma)$ to be the set of inverse descents of $\sigma$.

The inversion and major index statistics are defined by
\begin{equation}\label{eq:invmaj} \inv(\sigma)= \#\Inv(\sigma),\quad
    \maj(\sigma)=\sum_{i \in \Des(\sigma)} i.
\end{equation}
For instance, we have
\[\inv(\sigma)=5,\quad \maj(\sigma)=2\]
for $\sigma=(3,5,1,2,4)$ as above.
The inv and maj statistics are \emph{Mahonian}, meaning that
\begin{equation}
    \label{eq:mahonian}
\sum_{\sigma\in S_n} q^{\inv(\sigma)}=
\sum_{\sigma \in S_n} q^{\maj(\sigma)}=[n]_q!
\end{equation}
where $[n]_q!$ is the $q$-factorial
\[[n]_q!=[n]_q\cdots [1]_q,\quad
[k]_q=1+\cdots +q^{k-1}.\]

We also have the inversion and major index tables, whose total sum are the inversion number and major index respectively.
First we define $\invt(\sigma)=(a_1,...,a_n)$, where $a_j$ is the number of indices $j$ for which 
$(i,j) \in \Inv(\sigma)$. For the major
index table, define
a \emph{run} of $\sigma$ to be a maximal increasing sequence $(\sigma_i,...,\sigma_{i+l})$. 
We may decompose
$\sigma$ so that each $\sigma_i$ is in exactly one
run. We let $\majt(\sigma)=(b_1,...,b_n)$ so that if If $\sigma$ has $r$ runs, then define $b_i = r-k$ if $i$ appears in the $k$th run of $\sigma$. 
For instance, for $\sigma=(3,4,1,5,2)$ we would have
\[\invt(\sigma)=(2,3,0,0,0),\quad
\majt(\sigma)=(1,0,2,2,1).\]

We will denote the sets of all inversion
tables and major index tables by
$\invtabs_n$ and $\majtabs_n$ respectively, and call the elements of the latter \emph{descent compositions}.
For instance, we have
\[\majtabs_3=\left\{(0,0,0),(0,0,1),(0,1,0),(0,1,1),(1,0,1),(0,1,2)\right\}.\]
These represent the exponents of the Artin and Garsia-Stanton descent monomials, defined
in Section \ref{subsec:coinvariants}.

\subsection{Shuffle Theorem combinatorics}

\label{sec:parking}

We recall some definitions from \cite{haglund2008catalan} as well as \cite{haglund2005combinatoriala}, and deduce
a combinatorial corollary of the compositional Shuffle Theorem proved in \cite{carlsson2018proof}.
The main result we will need later is the statement of Proposition \ref{prop:ribbonhall}
at the end of the section. The reader who is willing to accept that statement may 
safely skip the other statements.

A Dyck path $\pi$ of length $n$ is a Northeastern lattice path from $(0,0)$ to $(n,n)$ that does not pass under the line $y = x$. The area sequence of $\pi$ is the $n$-tuple $\an = (a_1,\dots,a_n) \in \mathbb{Z}_{\geq 0}^n$, where $a_i$ is the number of complete boxes in the $i$th row, starting from $(0,0)$, between $\pi$ and the diagonal. Notice that $\pi$ uniquely determines an area sequence, and vice versa.
A parking function $P$ is a pair $(\an,\sigma)$ consisting of a area sequence $\an \in \mathbb{Z}_{\geq 0}^n$, together with a permutation $\sigma \in S_n$ with the property that $a_i +1 = a_{i+1} \implies \sigma_i < \sigma_{i+1}$. The set of parking functions is denoted by $\PF_n$. The area of a parking function is the number of boxes between the diagonal and the Dyck path, which is the same as the norm of the area sequence $|\an|$.
We denote the underlying Dyck path of 
$P$ by $\dyckpath(P)$.

Since there are no duplicate entries
of $\sigma$, we may 
refer to an entry of a parking function $P$ by the number $\sigma_i$ occupying it without ambiguity, as opposed to 
the row index. We will write $\lev(\sigma_j) = a_j$, the number of steps the box containing $\sigma_i$ is above the diagonal, starting with zero.
A \emph{dinv pair} of $P$ is a pair $(\sigma_i,\sigma_j)$ with $i < j$ satisfying either:
\begin{enumerate}
    \item $\sigma_i < \sigma_j, \quad a_i = a_i$, or
    \item $\sigma_i > \sigma_j, \quad a_i = a_j + 1$
\end{enumerate}
The total number of dinv pairs is denoted
$\dinv(P)$, whereas the collection of dinv pairs will be written in
bold $\dinvn(P)$. This is slightly different from the usual collection 
$\Dinv(P)$, which consists of the pairs of indices $(i,j)$ rather than the
labels $(\sigma_i,\sigma_j)$.

\begin{example}
  The parking function
\\
\begin{tikzpicture}[scale=.5]\draw[dotted] (0,0)--(9,9);\draw[black] (0,0)--(0,9);\draw[black] (0,0)--(9,0);\draw[black] (0,9)--(9,9);\draw[black] (9,0)--(9,9);\draw[black] (0,0)--(0,1);\node[scale=1] (2) at (0.500000,0.500000) {2};\draw[black] (0,1)--(1,1);\draw[black] (1,1)--(1,2);\node[scale=1] (4) at (1.500000,1.500000) {4};\draw[black] (1,2)--(1,3);\node[scale=1] (7) at (1.500000,2.500000) {7};\draw[black] (1,3)--(3,3);\draw[black] (3,3)--(3,4);\node[scale=1] (9) at (3.500000,3.500000) {9};\draw[black] (3,4)--(4,4);\draw[black] (4,4)--(4,5);\node[scale=1] (1) at (4.500000,4.500000) {1};\draw[black] (4,5)--(4,6);\node[scale=1] (5) at (4.500000,5.500000) {5};\draw[black] (4,6)--(4,7);\node[scale=1] (8) at (4.500000,6.500000) {8};\draw[black] (4,7)--(5,7);\draw[black] (5,7)--(5,8);\node[scale=1] (3) at (5.500000,7.500000) {3};\draw[black] (5,8)--(5,9);\node[scale=1] (6) at (5.500000,8.500000) {6};\end{tikzpicture}
\\
is given by the pair
$P=((0, 0, 1, 0, 0, 1,2,2,3), (2,4,7,9,1,5,8,3,6))$,
so $\area(P)=9$.  
We also have $\dinvn(P) = 
\{ (2, 4), (2, 9), (4, 9), (7, 1)) \}$, so that $\dinv(P) = 4$.

\end{example}

The (non-compositional) 
Shuffle Theorem \cite{haglund2005combinatorial} states that
\begin{thm}[The Shuffle Theorem \cite{carlsson2018proof}] We have
\[\nabla e_n \big|_{m_\mu}=
\sum_{P:\readword(P)\in \shuff(\mu)} t^{\area(P)} q^{\dinv(P)}\]
\end{thm}

On the left hand side, we have that taking the coefficient 
of $m_{1^n}$ at either $q=0$ or $t=0$ results in
$[n]_t!$ or $[n]_q!$ respectively. On the right hand side, it is clear that there are $n!$ parking functions with $\area(P)=0$,
and summing $q^{\dinv(P)}$ is equivalent to the sum over inversions in \eqref{eq:mahonian}.
The other side is less clear: denote the set of parking
functions with $\dinv(P)=0$ by $\PF_n^0$. Then the following is not difficult to show:
\begin{lem}
\label{lem:dinvzero}
A parking function $P$ is in $\PF_n^0$ if and only if
\begin{enumerate}
    \item \label{lem:item:consecutive} The area sequence is non-strictly increasing. In other words, the only consecutive East steps are in the top row. 
\item  We have $\sigma_i>\sigma_{j}$ whenever $i<j$ and $a_i=a_j$.
\end{enumerate}
\end{lem}

In fact the map $S_n\rightarrow \desc_n$ sending a permutation to its descent composition factors through a
bijection $S_n\leftrightarrow \PF_n^{0}$ which associates $P$ to its reading word, and which carries the major index to the area, for instance
\\
$(3, 1, 7, 5, 4, 2, 6)\rightarrow $
\begin{tikzpicture}[scale=.5]\draw[dotted] (0,0)--(7,7);\draw[black] (0,0)--(0,7);\draw[black] (0,0)--(7,0);\draw[black] (0,7)--(7,7);\draw[black] (7,0)--(7,7);\draw[black] (0,0)--(0,1);\node[scale=1] (6) at (0.500000,0.500000) {6};\draw[black] (0,1)--(1,1);\draw[black] (1,1)--(1,2);\node[scale=1] (2) at (1.500000,1.500000) {2};\draw[black] (1,2)--(1,3);\node[scale=1] (4) at (1.500000,2.500000) {4};\draw[black] (1,3)--(1,4);\node[scale=1] (5) at (1.500000,3.500000) {5};\draw[black] (1,4)--(1,5);\node[scale=1] (7) at (1.500000,4.500000) {7};\draw[black] (1,5)--(2,5);\draw[black] (2,5)--(2,6);\node[scale=1] (1) at (2.500000,5.500000) {1};\draw[black] (2,6)--(2,7);\node[scale=1] (3) at (2.500000,6.500000) {3};\end{tikzpicture}
$\rightarrow (3, 0, 4, 1, 2, 0, 3)$
\\
The first map reads the elements of $\sigma$ starting with the last entry in the $(0,0)$
position, and then adds the remaining entries in reverse order moving Northeast for descents, and North for ascents.
The second map forms a descent composition 
$\an \in \desc_n$ so that $a_i$ is the element of 
the area sequence in the entry containing $\sigma_i$. 
The main nontrivial step in showing that this determines a bijection between $S_n$ and
$\PF^0_n$ is to check that 
every $P\in \PF_n^0$ satisfies
part \ref{lem:item:consecutive} of the Lemma.

Following \cite{haglund2005combinatoriala},
given a composition $\alpha \vDash n$, we can consider the parking functions $P = (\an,\sigma)$ with the property that $a_1= a_{\alpha_1+1}=\dots=a_{n-\alpha_k+1} = 0$, i.e. the path must ``touch'' the diagonal at least at $\alpha_1, \alpha_1+\alpha_2$, etc., to be denoted $\alpha$-\emph{parking functions}, the set of which is denoted $\PF_\alpha$. Each $P\in \PF_{\alpha}$
correspond to $l$ individual parking functions
of size $\alpha_i$ called the ``blocks'' of $P$, provided we allow the labels to live outside of $\{1,...,\alpha_i\}$.
Define $r_k = |\{i:a_i = 0 \text{ and } \sum_{j=1}^{k-1}\alpha_j < i < \sum_{j = 1}^k \alpha_j|$, and set

$$ \doff_\alpha(P) = \sum_{k = 1}^{\ell(\alpha)} (\ell(\alpha) - k)r_k$$

\begin{example}
\label{ex:doff}
Consider the parking function
$P\in \PF_{1,2,6}$ given by

\begin{tikzpicture}
    \begin{tikzpicture}[scale=.5]\draw[dotted] (0,0)--(9,9);\draw[black] (0,0)--(0,9);\draw[black] (0,0)--(9,0);\draw[black] (0,9)--(9,9);\draw[black] (9,0)--(9,9);\draw[black] (0,0)--(0,1);\node[scale=1] (2) at (0.500000,0.500000) {2};\draw[black] (0,1)--(1,1);\draw[black] (1,1)--(1,2);\node[scale=1] (4) at (1.500000,1.500000) {4};\draw[black] (1,2)--(1,3);\node[scale=1] (7) at (1.500000,2.500000) {7};\draw[black] (1,3)--(3,3);\draw[black] (3,3)--(3,4);\node[scale=1] (9) at (3.500000,3.500000) {9};\draw[black] (3,4)--(4,4);\draw[black] (4,4)--(4,5);\node[scale=1] (1) at (4.500000,4.500000) {1};\draw[black] (4,5)--(4,6);\node[scale=1] (5) at (4.500000,5.500000) {5};\draw[black] (4,6)--(4,7);\node[scale=1] (8) at (4.500000,6.500000) {8};\draw[black] (4,7)--(5,7);\draw[black] (5,7)--(5,8);\node[scale=1] (3) at (5.500000,7.500000) {3};\draw[black] (5,8)--(5,9);\node[scale=1] (6) at (5.500000,8.500000) {6};\end{tikzpicture}
\end{tikzpicture}

\noindent
Then we have $\dinv(P)=4$, and
 $\doff_{(1,2,6)}(\dyckpath(P)) = 
1\cdot 2+1\cdot 1+2\cdot 0=3$.

\end{example}

Then one equivalent form of the compositional shuffle theorem
(see \cite{haglund2005combinatoriala}) states

\begin{thm}[Compositional Shuffle Theorem \cite{carlsson2018proof}]
We have that
\begin{equation}\label{eq:cmpshuff}
\nabla B_{\alpha} \big|_{m_\mu}=
\sum_{P\in \PF_{\alpha}:\ \readword(P)\in \shuff(\mu)} t^{\area(P)} q^{\dinv(P)+\doff_{\alpha}(\dyckpath(P))}
\end{equation}
where $B_\alpha$ is a certain symmetric function
associated to a composition $\alpha$. When
$\alpha=\revla=(\lambda_l,...,\lambda_1)$,
is a partition in reverse order, we have
\[B_{\alpha}=q^{n(\lambda)}
\omega \Ht_{\lambda}[X;q,t]\big|_{t=q^{-1},q=0}.\]
\end{thm}

We also have the ``odd'' version, which says that we obtain
$\omega \nabla B_{\alpha}$ by
replacing $\shuff(\mu)$ with 
$\shuff'(\mu)$.

\begin{cor}
The Hall-Littlewood polynomial is given by
\[\Ht_{\lambda'}[X;t]\big|_{m_\mu}=
\sum_{P\in \PF^{0}_{\alpha}:
\readword(P)\in \shuff(\mu)} t^{\area(P)}
\]
where $\alpha=\revla$, and
$\PF_{\alpha}^0\subset \PF_{\alpha}$ is the subcollection of parking functions $P$ which minimize the $q$-degree, i.e.
$\dinv(P)+\doff_{\alpha}(\dyckpath(P))=
n(\lambda)$.
\end{cor}

\begin{proof}

Take the lowest degree terms in $q$ of both sides
of both sides of \eqref{eq:cmpshuff}, and
use the fact that $\nabla$ is triangular in the modified Hall-Littlewood basis, with lowest degree terms in $q$ appearing on the diagonal.    
\end{proof}

\begin{example}
\label{ex:hallparking}
    We have that
\[\Ht_{(2,1,1)}[X;t]=
m_{{4}}+ \left( {t}^{2}+t+1 \right) m_{{3,1}}+ \left( 2
{t}^{2}+t+1 \right) m_{{2,2}}+\]
\[\left( {t}^{3}+3{t}^{2}+2
+1 \right) m_{{2,1,1}}+ \left(3{t}^{3}+5{t}^{2}+3t+1
 \right) m_{{1,1,1,1}}    \]
The parking functions in $\PF^0_{1,3}$
are given by

\begin{tikzpicture}[scale=.5]\draw[dotted] (0,0)--(4,4);\draw[black] (0,0)--(0,4);\draw[black] (0,0)--(4,0);\draw[black] (0,4)--(4,4);\draw[black] (4,0)--(4,4);\draw[black] (0,0)--(0,1);\node[scale=1] (2) at (0.500000,0.500000) {2};\draw[black] (0,1)--(1,1);\draw[black] (1,1)--(1,2);\node[scale=1] (1) at (1.500000,1.500000) {1};\draw[black] (1,2)--(1,3);\node[scale=1] (3) at (1.500000,2.500000) {3};\draw[black] (1,3)--(1,4);\node[scale=1] (4) at (1.500000,3.500000) {4};\end{tikzpicture}
\begin{tikzpicture}[scale=.5]\draw[dotted] (0,0)--(4,4);\draw[black] (0,0)--(0,4);\draw[black] (0,0)--(4,0);\draw[black] (0,4)--(4,4);\draw[black] (4,0)--(4,4);\draw[black] (0,0)--(0,1);\node[scale=1] (3) at (0.500000,0.500000) {3};\draw[black] (0,1)--(1,1);\draw[black] (1,1)--(1,2);\node[scale=1] (1) at (1.500000,1.500000) {1};\draw[black] (1,2)--(1,3);\node[scale=1] (2) at (1.500000,2.500000) {2};\draw[black] (1,3)--(1,4);\node[scale=1] (4) at (1.500000,3.500000) {4};\end{tikzpicture}
\begin{tikzpicture}[scale=.5]\draw[dotted] (0,0)--(4,4);\draw[black] (0,0)--(0,4);\draw[black] (0,0)--(4,0);\draw[black] (0,4)--(4,4);\draw[black] (4,0)--(4,4);\draw[black] (0,0)--(0,1);\node[scale=1] (4) at (0.500000,0.500000) {4};\draw[black] (0,1)--(1,1);\draw[black] (1,1)--(1,2);\node[scale=1] (1) at (1.500000,1.500000) {1};\draw[black] (1,2)--(1,3);\node[scale=1] (2) at (1.500000,2.500000) {2};\draw[black] (1,3)--(1,4);\node[scale=1] (3) at (1.500000,3.500000) {3};\end{tikzpicture}
\begin{tikzpicture}[scale=.5]\draw[dotted] (0,0)--(4,4);\draw[black] (0,0)--(0,4);\draw[black] (0,0)--(4,0);\draw[black] (0,4)--(4,4);\draw[black] (4,0)--(4,4);\draw[black] (0,0)--(0,1);\node[scale=1] (2) at (0.500000,0.500000) {2};\draw[black] (0,1)--(1,1);\draw[black] (1,1)--(1,2);\node[scale=1] (1) at (1.500000,1.500000) {1};\draw[black] (1,2)--(1,3);\node[scale=1] (4) at (1.500000,2.500000) {4};\draw[black] (1,3)--(2,3);\draw[black] (2,3)--(2,4);\node[scale=1] (3) at (2.500000,3.500000) {3};\end{tikzpicture}
\begin{tikzpicture}[scale=.5]\draw[dotted] (0,0)--(4,4);\draw[black] (0,0)--(0,4);\draw[black] (0,0)--(4,0);\draw[black] (0,4)--(4,4);\draw[black] (4,0)--(4,4);\draw[black] (0,0)--(0,1);\node[scale=1] (3) at (0.500000,0.500000) {3};\draw[black] (0,1)--(1,1);\draw[black] (1,1)--(1,2);\node[scale=1] (1) at (1.500000,1.500000) {1};\draw[black] (1,2)--(1,3);\node[scale=1] (4) at (1.500000,2.500000) {4};\draw[black] (1,3)--(2,3);\draw[black] (2,3)--(2,4);\node[scale=1] (2) at (2.500000,3.500000) {2};\end{tikzpicture}
\begin{tikzpicture}[scale=.5]\draw[dotted] (0,0)--(4,4);\draw[black] (0,0)--(0,4);\draw[black] (0,0)--(4,0);\draw[black] (0,4)--(4,4);\draw[black] (4,0)--(4,4);\draw[black] (0,0)--(0,1);\node[scale=1] (3) at (0.500000,0.500000) {3};\draw[black] (0,1)--(1,1);\draw[black] (1,1)--(1,2);\node[scale=1] (2) at (1.500000,1.500000) {2};\draw[black] (1,2)--(1,3);\node[scale=1] (4) at (1.500000,2.500000) {4};\draw[black] (1,3)--(2,3);\draw[black] (2,3)--(2,4);\node[scale=1] (1) at (2.500000,3.500000) {1};\end{tikzpicture}

\begin{tikzpicture}[scale=.5]\draw[dotted] (0,0)--(4,4);\draw[black] (0,0)--(0,4);\draw[black] (0,0)--(4,0);\draw[black] (0,4)--(4,4);\draw[black] (4,0)--(4,4);\draw[black] (0,0)--(0,1);\node[scale=1] (4) at (0.500000,0.500000) {4};\draw[black] (0,1)--(1,1);\draw[black] (1,1)--(1,2);\node[scale=1] (1) at (1.500000,1.500000) {1};\draw[black] (1,2)--(1,3);\node[scale=1] (3) at (1.500000,2.500000) {3};\draw[black] (1,3)--(2,3);\draw[black] (2,3)--(2,4);\node[scale=1] (2) at (2.500000,3.500000) {2};\end{tikzpicture}
\begin{tikzpicture}[scale=.5]\draw[dotted] (0,0)--(4,4);\draw[black] (0,0)--(0,4);\draw[black] (0,0)--(4,0);\draw[black] (0,4)--(4,4);\draw[black] (4,0)--(4,4);\draw[black] (0,0)--(0,1);\node[scale=1] (4) at (0.500000,0.500000) {4};\draw[black] (0,1)--(1,1);\draw[black] (1,1)--(1,2);\node[scale=1] (2) at (1.500000,1.500000) {2};\draw[black] (1,2)--(1,3);\node[scale=1] (3) at (1.500000,2.500000) {3};\draw[black] (1,3)--(2,3);\draw[black] (2,3)--(2,4);\node[scale=1] (1) at (2.500000,3.500000) {1};\end{tikzpicture}
\begin{tikzpicture}[scale=.5]\draw[dotted] (0,0)--(4,4);\draw[black] (0,0)--(0,4);\draw[black] (0,0)--(4,0);\draw[black] (0,4)--(4,4);\draw[black] (4,0)--(4,4);\draw[black] (0,0)--(0,1);\node[scale=1] (3) at (0.500000,0.500000) {3};\draw[black] (0,1)--(1,1);\draw[black] (1,1)--(1,2);\node[scale=1] (2) at (1.500000,1.500000) {2};\draw[black] (1,2)--(2,2);\draw[black] (2,2)--(2,3);\node[scale=1] (1) at (2.500000,2.500000) {1};\draw[black] (2,3)--(2,4);\node[scale=1] (4) at (2.500000,3.500000) {4};\end{tikzpicture}
\begin{tikzpicture}[scale=.5]\draw[dotted] (0,0)--(4,4);\draw[black] (0,0)--(0,4);\draw[black] (0,0)--(4,0);\draw[black] (0,4)--(4,4);\draw[black] (4,0)--(4,4);\draw[black] (0,0)--(0,1);\node[scale=1] (4) at (0.500000,0.500000) {4};\draw[black] (0,1)--(1,1);\draw[black] (1,1)--(1,2);\node[scale=1] (2) at (1.500000,1.500000) {2};\draw[black] (1,2)--(2,2);\draw[black] (2,2)--(2,3);\node[scale=1] (1) at (2.500000,2.500000) {1};\draw[black] (2,3)--(2,4);\node[scale=1] (3) at (2.500000,3.500000) {3};\end{tikzpicture}
\begin{tikzpicture}[scale=.5]\draw[dotted] (0,0)--(4,4);\draw[black] (0,0)--(0,4);\draw[black] (0,0)--(4,0);\draw[black] (0,4)--(4,4);\draw[black] (4,0)--(4,4);\draw[black] (0,0)--(0,1);\node[scale=1] (4) at (0.500000,0.500000) {4};\draw[black] (0,1)--(1,1);\draw[black] (1,1)--(1,2);\node[scale=1] (3) at (1.500000,1.500000) {3};\draw[black] (1,2)--(2,2);\draw[black] (2,2)--(2,3);\node[scale=1] (1) at (2.500000,2.500000) {1};\draw[black] (2,3)--(2,4);\node[scale=1] (2) at (2.500000,3.500000) {2};\end{tikzpicture}
\begin{tikzpicture}[scale=.5]\draw[dotted] (0,0)--(4,4);\draw[black] (0,0)--(0,4);\draw[black] (0,0)--(4,0);\draw[black] (0,4)--(4,4);\draw[black] (4,0)--(4,4);\draw[black] (0,0)--(0,1);\node[scale=1] (4) at (0.500000,0.500000) {4};\draw[black] (0,1)--(1,1);\draw[black] (1,1)--(1,2);\node[scale=1] (3) at (1.500000,1.500000) {3};\draw[black] (1,2)--(2,2);\draw[black] (2,2)--(2,3);\node[scale=1] (2) at (2.500000,2.500000) {2};\draw[black] (2,3)--(3,3);\draw[black] (3,3)--(3,4);\node[scale=1] (1) at (3.500000,3.500000) {1};\end{tikzpicture}
\\
Summing $t^{\area(P)}$ gives the desired coefficient of $m_{1,1,1,1}$.

\end{example}

We now have an extension of Lemma \ref{lem:dinvzero}.
\begin{prop}
\label{prop:mindd}
Let $\alpha=\revla$ where $\lambda$ is a Young diagram.
Then the elements of $P\in\PF_{\alpha}^0$
are determined by the following properties:
\begin{enumerate}
\item \label{item:isribbon} Within each block
$P^{(k)}$ of size $\alpha_k$, the area sequence
$a_i$ is non-strictly increasing.
whenever $a_{i}=a_{i+1}$ in a given block,
we have $\sigma_i>\sigma_{i+1}$.
\item The dinv pairs $(\sigma_i,\sigma_j)\in \dinvn(P)$ all have $(\sigma_i,\sigma_j)$ occurring accross different blocks.
Moreover, for each pair of different blocks,
$P^{(b)},P^{(c)}$ with $b<c$ and each 
$\sigma_i \in P^{(b)}$, there is exactly one dinv pair $(\sigma_i,\sigma_j)$ with $\sigma_j\in P^{(c)}$.
\item There are no dinv pairs $(\sigma_i,\sigma_j)$ for which $a_i=0$.
\end{enumerate}
\end{prop}

Before proving the proposition, we will find it useful in what follows to describe the corressponding elements in terms of certain ribbon-shaped tableaux:
\begin{defn}
Let $\ribtabs_{n}$ be the collection of first-quadrant skew diagrams with no $2\times 2$ shapes (ribbons), filled with
distinct integers increasing in the 
North/East directions.
Given a composition
$\alpha=(\alpha_1,...,\alpha_l)$, 
let $\ribtabs_{\lambda}$ 
be the collection of $l$-tuples of
ribbon-shaped tableaux 
$\tilde{T}=(T_1,...,T_l)$ with 
$|T_i|=\lambda_i$, and with 
lowest boxes aligned.
Given a box $x\in \tilde{T}$, the corresponding entry is denoted by $\sigma(x)$.
\end{defn}

Then we have an injective map
$\ribtabs_{\lambda}\rightarrow \PF_{\alpha}$
for $\alpha=\revla$ which reflects the diagram over the $y$-axis 
and then shears all East steps
into Northeast steps. 
For instance, the tableau $\tilde{T}$ given by

\begin{tikzpicture}
    \filldraw[thin,fill=white]
    (3.500000,0.000000)--(4.000000,0.000000)--(4.000000,0.500000)--(3.500000
    ,0.500000)--(3.500000,0.000000);
    \node at (3.750000,0.250000) {2};
    \filldraw[thin,fill=white]
    (2.500000,0.000000)--(3.000000,0.000000)--(3.000000,0.500000)--(2.500000
    ,0.500000)--(2.500000,0.000000);
    \node at (2.750000,0.250000) {4};
    \filldraw[thin,fill=white]
    (2.500000,0.500000)--(3.000000,0.500000)--(3.000000,1.000000)--(2.500000
    ,1.000000)--(2.500000,0.500000);
    \node at (2.750000,0.750000) {7};
    \filldraw[thin,fill=white]
    (1.500000,0.000000)--(2.000000,0.000000)--(2.000000,0.500000)--(1.500000
    ,0.500000)--(1.500000,0.000000);
    \node at (1.750000,0.250000) {9};
    \filldraw[thin,fill=white]
    (1.000000,0.000000)--(1.500000,0.000000)--(1.500000,0.500000)--(1.000000
    ,0.500000)--(1.000000,0.000000);
    \node at (1.250000,0.250000) {1};
    \filldraw[thin,fill=white]
    (1.000000,0.500000)--(1.500000,0.500000)--(1.500000,1.000000)--(1.000000
    ,1.000000)--(1.000000,0.500000);
    \node at (1.250000,0.750000) {5};
    \filldraw[thin,fill=white]
    (1.000000,1.000000)--(1.500000,1.000000)--(1.500000,1.500000)--(1.000000
    ,1.500000)--(1.000000,1.000000);
    \node at (1.250000,1.250000) {8};
    \filldraw[thin,fill=white]
    (0.500000,1.000000)--(1.000000,1.000000)--(1.000000,1.500000)--(0.500000
    ,1.500000)--(0.500000,1.000000);
    \node at (0.750000,1.250000) {3};
    \filldraw[thin,fill=white]
    (0.500000,1.500000)--(1.000000,1.500000)--(1.000000,2.000000)--(0.500000
    ,2.000000)--(0.500000,1.500000);
    \node at (0.750000,1.750000) {6};
\end{tikzpicture}
\\
corresponds to the parking function from Example \ref{ex:doff}.
It is straightforward to see that the image of this map is precisely those parking functions
$P\in \PF_{\lambda}$ satisfying condition
\ref{item:isribbon} from Proposition \ref{prop:mindd}.

We can easily interpret the parking function statistics under this map in terms of $\tilde{T}$. The corresponding dinv statistic is 
the number of pairs 
of entries $(a,b)$, taken either from the
same component $T_i$ or from different ones, satisfying one of the 
following:
\\
$i.\  \qquad \vcenter{ \hbox{
\begin{ytableau}
    a & \none & \none & \none & b 
\end{ytableau} }}
$, $\quad a>b$
\\
$
ii. \qquad \vcenter{ \hbox{
\begin{ytableau}
    \none & \none & \none & \none & b \\
    a & \none & \none & \none & *(gray)
\end{ytableau}
}},\quad a<b
$
\\

 Given a box $x \in \tilde{T}$, let $\height(x)$ be the height of the box, so that 
    the bottom row has height $0$. The 
    height of a ribbon $T_i$ is $\height(T_i) = \max_{x \in T_i}\{\height(x)\}$. We next
    have that the $\doff_{\alpha}$ 
    statistic corresponds to $\doff_{\lambda}(\tilde{T})$,
where
    \begin{equation}
        \doff_{\lambda}(\tilde{T}) = \sum_{i = 1}^{\ell({\lambda})} (i-1)\bigg(\#\bigg\{x \in T_i : \height(x) = 0\bigg\}\bigg) = \sum_{i < j} \#\big\{ x \in T_j : 
        \height(x) = 0 \big\}
    \end{equation}

Last, the area sequence is the result of
setting $a_k$ to be the height of the box containing the entry $k$.
The reading word of $\readword(\tilde{T})$ corresponds to the common reading order from top to bottom, right to left.

Then interpreting the other two conditions shows that the following ribbons are in bijection
with $\PF_{\alpha}^0$.
\begin{defn}
\label{def:minrib}
Let $\lambda$ be a Young diagram, and define $\mathcal{R}^0_{\lambda}\subset \mathcal{R}_{\lambda}$ be the subcollection
of tableaux $\tilde{T}$ for which
\begin{enumerate}
    \item \label{item:minribdinv} All dinv-pairs 
    $(\sigma(x),\sigma(y))$ occur in different
    tableaux $x\in T_i$,
    and $y\in T_j$ for $i<j$.
Each $y \in T_j$ appears in exactly one dinv pair $(\sigma(x),\sigma(y))$ with $x \in T_i$ for each $i<j$.
\item \label{itemi:minribbottom} There are no dinv pairs $(\sigma(x),\sigma(y))$ with the rightmost box $y$ in the 
bottom row. In other words, the bottom row is in increasing order.
\end{enumerate}
\end{defn}
For instance, the following is an element of 
$\ribtabs^0_{10,9,4}$:
\\
\begin{tikzpicture}
 \filldraw[thin,fill=white]
    (8.500000,0.000000)--(9.000000,0.000000)--(9.000000,0.500000)--(8.500000
    ,0.500000)--(8.500000,0.000000);
    \node at (8.750000,0.250000) {19};
    \filldraw[thin,fill=white]
    (8.000000,0.000000)--(8.500000,0.000000)--(8.500000,0.500000)--(8.000000
    ,0.500000)--(8.000000,0.000000);
    \node at (8.250000,0.250000) {9};
    \filldraw[thin,fill=white]
    (8.000000,0.500000)--(8.500000,0.500000)--(8.500000,1.000000)--(8.000000
    ,1.000000)--(8.000000,0.500000);
    \node at (8.250000,0.750000) {14};
    \filldraw[thin,fill=white]
    (7.500000,0.500000)--(8.000000,0.500000)--(8.000000,1.000000)--(7.500000
    ,1.000000)--(7.500000,0.500000);
    \node at (7.750000,0.750000) {1};
    \filldraw[thin,fill=white]
    (6.500000,0.000000)--(7.000000,0.000000)--(7.000000,0.500000)--(6.500000
    ,0.500000)--(6.500000,0.000000);
    \node at (6.750000,0.250000) {6};
    \filldraw[thin,fill=white]
    (6.500000,0.500000)--(7.000000,0.500000)--(7.000000,1.000000)--(6.500000
    ,1.000000)--(6.500000,0.500000);
    \node at (6.750000,0.750000) {11};
    \filldraw[thin,fill=white]
    (6.500000,1.000000)--(7.000000,1.000000)--(7.000000,1.500000)--(6.500000
    ,1.500000)--(6.500000,1.000000);
    \node at (6.750000,1.250000) {22};
    \filldraw[thin,fill=white]
    (6.000000,1.000000)--(6.500000,1.000000)--(6.500000,1.500000)--(6.000000
    ,1.500000)--(6.000000,1.000000);
    \node at (6.250000,1.250000) {21};
    \filldraw[thin,fill=white]
    (5.500000,1.000000)--(6.000000,1.000000)--(6.000000,1.500000)--(5.500000
    ,1.500000)--(5.500000,1.000000);
    \node at (5.750000,1.250000) {18};
    \filldraw[thin,fill=white]
    (5.000000,1.000000)--(5.500000,1.000000)--(5.500000,1.500000)--(5.000000
    ,1.500000)--(5.000000,1.000000);
    \node at (5.250000,1.250000) {15};
    \filldraw[thin,fill=white]
    (4.500000,1.000000)--(5.000000,1.000000)--(5.000000,1.500000)--(4.500000
    ,1.500000)--(4.500000,1.000000);
    \node at (4.750000,1.250000) {12};
    \filldraw[thin,fill=white]
    (4.500000,1.500000)--(5.000000,1.500000)--(5.000000,2.000000)--(4.500000
    ,2.000000)--(4.500000,1.500000);
    \node at (4.750000,1.750000) {20};
    \filldraw[thin,fill=white]
    (4.000000,1.500000)--(4.500000,1.500000)--(4.500000,2.000000)--(4.000000
    ,2.000000)--(4.000000,1.500000);
    \node at (4.250000,1.750000) {17};
    \filldraw[thin,fill=white]
    (3.000000,0.000000)--(3.500000,0.000000)--(3.500000,0.500000)--(3.000000
    ,0.500000)--(3.000000,0.000000);
    \node at (3.250000,0.250000) {2};
    \filldraw[thin,fill=white]
    (3.000000,0.500000)--(3.500000,0.500000)--(3.500000,1.000000)--(3.000000
    ,1.000000)--(3.000000,0.500000);
    \node at (3.250000,0.750000) {8};
    \filldraw[thin,fill=white]
    (3.000000,1.000000)--(3.500000,1.000000)--(3.500000,1.500000)--(3.000000
    ,1.500000)--(3.000000,1.000000);
    \node at (3.250000,1.250000) {10};
    \filldraw[thin,fill=white]
    (3.000000,1.500000)--(3.500000,1.500000)--(3.500000,2.000000)--(3.000000
    ,2.000000)--(3.000000,1.500000);
    \node at (3.250000,1.750000) {16};
    \filldraw[thin,fill=white]
    (2.500000,1.500000)--(3.000000,1.500000)--(3.000000,2.000000)--(2.500000
    ,2.000000)--(2.500000,1.500000);
    \node at (2.750000,1.750000) {7};
    \filldraw[thin,fill=white]
    (2.000000,1.500000)--(2.500000,1.500000)--(2.500000,2.000000)--(2.000000
    ,2.000000)--(2.000000,1.500000);
    \node at (2.250000,1.750000) {4};
    \filldraw[thin,fill=white]
    (1.500000,1.500000)--(2.000000,1.500000)--(2.000000,2.000000)--(1.500000
    ,2.000000)--(1.500000,1.500000);
    \node at (1.750000,1.750000) {3};
    \filldraw[thin,fill=white]
    (1.500000,2.000000)--(2.000000,2.000000)--(2.000000,2.500000)--(1.500000
    ,2.500000)--(1.500000,2.000000);
    \node at (1.750000,2.250000) {23};
    \filldraw[thin,fill=white]
    (1.000000,2.000000)--(1.500000,2.000000)--(1.500000,2.500000)--(1.000000
    ,2.500000)--(1.000000,2.000000);
    \node at (1.250000,2.250000) {13};
    \filldraw[thin,fill=white]
    (0.500000,2.000000)--(1.000000,2.000000)--(1.000000,2.500000)--(0.500000
    ,2.500000)--(0.500000,2.000000);
    \node at (0.750000,2.250000) {5};
\end{tikzpicture}

\begin{proof}[Proof of Proposition \ref{prop:mindd}]
We will prove that 
the set of those tableaux $\tilde{T} \in \ribtabs_\lambda$ which minimize
$\dinv(\tilde{T}) + \doff_{\lambda}(\tilde{T})$
is equal to 
$\ribtabs_{\lambda}^0$.
Suppose that $\tilde{T}$ is minimizes the
$\dinv+\doff_{\lambda}$ statistic,
and let us rewrite the sum
according to contribution by pairs $(\lambda_i,\lambda_j)$:
    \begin{equation}
    \label{eq:doffij}
        \dinv(\tilde{T}) + \doff_{\lambda}(\tilde{T}) = \sum_{i < j} \bigg( \#\big\{ x \in T_j : \height(x) = 0 \big\} + \dinv_{i,j}(\tilde{T})\} \bigg)
    \end{equation}
    Here $\dinv_{i,j}(\tilde{T})$ is the number of dinv pairs $(\sigma(x),\sigma(y))$ with $x \in T_i,y \in T_j$. We first show that the contribution of each pair $(i,j)$ to 
    the sum in \eqref{eq:doffij}
    is at least $\lambda_j$ 
    (and therefore exactly $\lambda_j$).

We have two cases: for a given $i<j$, suppose 
$\height(T_i) \geq  \height(T_j)$. This means for each square $y \in T_j$, there is a square $x \in T_i$ in the same row. Then, we have two cases:
    \begin{itemize}
        \item If $\height(y) = 0$, then $a$ contributes $1$ to $\doff_\lambda(T)$.
        \item If a square is not in the bottom row, then the following occurs:

        \begin{center}
            \begin{ytableau}
                x & \none & \none & \none & y \\
                z
            \end{ytableau}
        \end{center}
        where necessarily $\sigma(x) > \sigma(z)$. Then, if $\sigma(x) < \sigma(y)$, we have that $\sigma(y) > \sigma(z)$, so $(\sigma(z),\sigma(y)) \in \dinvn_{i,j}(\tilde{T})$. If $\sigma(x) > \sigma(y)$, then $(\sigma(x),\sigma(y)) \in \dinvn_{i,j}(\tilde{T})$.
    \end{itemize}
    In either case, $a$ contributes at least $1$ to $\dinv(T) + \doff_\lambda(T)$.

    In the other case, if $\height(T_i) < \height(T_j)$, then for each $x \in T_i$, there is a square in the same row $y \in T_j$, but also a square in the row above. Then, for each square $x \in T_i$, find the leftmost square in $y \in T_j$ of the same row. One of two things may occur:

    \begin{itemize}
        \item The entry $y$ is the corner of a ribbon:
        \begin{center}
        \begin{ytableau}
            \none & \none & \none & z & \none \\
            x & \none & \none & y & \dots
        \end{ytableau}
        \end{center}
        \item There is only one square in the same row:
        \begin{center}
        \begin{ytableau}
             \none & \none & \none & \none & z \\
             x & \none & \none & \none & y \\
             \none & \none & \none & \none & \vdots 
        \end{ytableau}
        \end{center}        
    \end{itemize}

    If $\sigma(x) < \sigma(y)$, then $(\sigma(x),\sigma(y)) \in \dinvn_{i,j}(\tilde{T})$. Otherwise, if $\sigma(x) > \sigma(y)$, since the tableau is column decreasing, we have that $\sigma(z) > \sigma(x) > \sigma(y) $, so that $(\sigma(z),\sigma(y)) \in \dinvn_{i,j}(\tilde{T})$. Then, every square of $T_i$ will contribute at least $1$ to $(\sigma(x),\sigma(y)) \in \dinvn_{i,j}(\tilde{T})$, so that the contribution is at least $\lambda_i \geq \lambda_j$.

Now the sum of $\lambda_j$ over pairs $i<j$ is $n(\lambda)$, so the inequality must be tight, and the above contribution must be exactly $\lambda_j$. Since there is at least one box in the bottom row of $T_j$ contributing to $\doff_{\lambda}$, 
it follows that the second case of $\height(T_i)<\height(T_j)$ can never happen.
But in the first case we showed every box in
$T_j$
satisfies item \ref{item:minribdinv} of 
Definition \ref{def:minrib}.
The second item is also clear.

\end{proof}

Putting this together, we have
\begin{prop}
\label{prop:ribbonhall}    
Let $\alpha=\revla$ as above. Then we have that
\begin{align}
\label{eq:ribhallprop}
\begin{split}
\Ht_{\lambda'}[X;t]\big|_{m_\mu}&=
\sum_{\tilde{T}\in \ribtabs^{0}_{\lambda}:
\readword(\tilde{T})\in 
\shuff(\mu)} t^{\area(\tilde{T})} \\
\omega \Ht_{\lambda'}[X;t]\big|_{m_\mu}&=
\sum_{\tilde{T}\in \ribtabs^{0}_{\lambda}:
\readword(\tilde{T})\in 
\shuff'(\mu)} t^{\area(\tilde{T})}
\end{split}
\end{align}
\end{prop}
In particular,
the size of $\mathcal{R}^0_\lambda$ is the
multinomial coefficient 
\begin{equation}
\label{eq:ribbonmulti}
|\mathcal{R}^0_{\lambda}|=\binom{n}
{\lambda'_1, ...,\lambda'_l}
\end{equation}
in the parts of the transposed partition.
For instance, the twelve elements of $\mathcal{R}^0_{3,1}$ corresponding to the parking functions from the example are
are the ones corresponding to the parking
functions in Example \ref{ex:hallparking}.

\subsection{The Artin and descent basis}

\label{subsec:coinvariants}

Let $R_n=\C[\xn]/I_n$ 
where $\xn=(x_1,...,x_n)$, and
\[I_n=\left(e_1(\xn),...,e_n(\xn)\right)\]
is the ideal generated by symmetric polynomials
with vanishing constant term.
The dimension is well known to be $n!$, and the
Hilbert series is given by the $q$-factorial
$\hs_q (R_n)=[n]_q!$. 

Given a permutation $\sigma \in S_n$, we define two monomials in the variables $\xn$:\
\[f_{\sigma}(\xn)=\prod_{i<j:\sigma_i>\sigma_j}
x_{\sigma_i},\quad g_{\sigma}(\xn)=
\prod_{i:\sigma_i>\sigma_{i+1}} x_{\sigma_1}\cdots x_{\sigma_i}.\]
Both sets of monomials form a vector
space basis of $R_n(\xn)$, known as the
Artin basis and Garsia-Stanton descent basis
respectively \cite{garsia1993graded}.
The exponents
\[f_\sigma(\xn)= x_1^{a_1}\cdots x_n^{a_n},\quad
g_\sigma(\xn)= x_1^{b_1}\cdots x_n^{b_n}\]
are precisely the inversion and major index tables \[\invt(\sigma)=(a_1,...,a_n),\quad \majt(\sigma)=(b_1,...,b_n).\]
Reading off the degrees of these coefficients, we recover 
the statement that $\inv,\maj$ are Mahonian.

These two bases may be uniquely characterized as being leading monomials
with respect to different orderings
on the monomials $\xn^\an$. In the case
of the Artin basis, the ordering is simply the lex order, $\xn^\an \leq_{lex} \xn^\bn$
provided that $\an \leq_{lex} \bn$ in the lexicographic ordering on the variables.
We may then say that
$\lm(R_n)=\{\xn^\an:\an \in \invtabs_n\}$
where $\lm_{\leq_{lex}}(R_n)=\lm_{\leq_{lex}}(\Q[\xn]/I_n)$ is the
collection of monomials which are not
in $\lm_{\leq_{lex}}(I_n)$. In other words, the monomials of the for $\xn^\an$ for $\an \in \invtabs_n$ are the ones which are not divisible by the leading monomial of any element of the Gro\"{b}ner basis of $I_n$.

The descent monomials may be defined similarly, for for a different order known as the descent order:
it is defined by $\xn^{\an}\leq_{des} \xn^{\bn}$ if
\begin{enumerate}
    \item $\sort(\an,>) <_{lex} \sort(\bn,>)$ or
\item $\sort(\an)=\sort(\bn)$ and
$\an \leq_{lex} \bn$.
\end{enumerate}
In other words, we first sort $\an$ in descending order to produce a partition
with some trailing zeros, and compare the result in the lex order. We then apply the lex order in $\an,\bn$ to break ties.
Then $\leq_{des}$ is not a monomial order for the purposes of Gr\"{o}bner bases, 
but it satisfies the following useful property: if $S\subset [n]$ is a subset with complement $T$, then we have
\begin{equation}
\label{eq:desprop}
\an|_S \leq_{des} \bn|_S,\ 
\an|_T \leq_{des} \bn|_T
\Rightarrow \an \leq_{des} \bn.
\end{equation}
Moreover, it still makes sense to talk about leading monomials, and we have
$\lm_{\leq_{des}}(R_n)=\{\xn^\an:\an \in \majtabs_n\}$.
In \cite{AllenDescent}, E. Allen gave an explicit algorithm for reducing any monomial in this order.

\subsection{The Garsia-Procesi module}

\label{subsec:garsia1992certain}

Let $a = (a_1,...,a_n) \in \mathbb{Z}_{\geq 0}^n$, and denote $x^a = x_1^{a_1}...x_n^{a_n}$. In \cite{AllenDescent}, it was shown that the collection of monomials $ \{ x^{\text{majt}(\tau)} : \tau \in S_n \}$ is a basis for the coinvariant algebra 

$$ R_n = \frac{\mathbb{C}[x_1,...,x_n]}{I_n}$$

where $I_n = \langle e_1(\mathbf{x}),...,e_n(\mathbf{x})\rangle$, viewed as a $\mathbb{C}$-vector space. The coinvariant algebra is well known to be isomorphic as $S_n$-modules to the cohomology ring of the flag variety, $H^*(Fl_n)$. 

There is a sort of generalization of the coinvariant algebra to partitions of $n$. Let $\lambda \vdash n$, and denote the conjugate partition by $\lambda' = (\lambda'_1 \geq ... \geq \lambda'_n \geq 0)$, where we pad $\lambda'$ with $0$'s until we reach a tuple of length $n$. Denote $p_m^n(\lambda) := \lambda'_n + ... + \lambda'_{n-m+1}$. If $S \subset \{x_1,...,x_n\}$, denote

$$ e_d(S) = \sum_{\substack{i_1 < ... < i_d \\ x_{i_j} \in S}}x_{i_1}...x_{i_d}$$
to be the sum of all squarefree monomials of degree $d$ in $S$. Then, the \emph{Tanisaki Ideal} $I_\lambda$ is defined to be

$$ I_\lambda = \bigg\langle e_d(S) : S \subseteq \{x_1,...,x_n\}, d > |S|-p^n_{|S|}(\lambda) \bigg\rangle$$

and the \emph{Garsia-Procesi module} $R_\lambda$ (\cite{garsia1992certain}) is defined to be $R_\lambda := \mathbb{C}[x_1,...,x_n]/I_\lambda$. The corresponding geometric object is known as the \emph{Springer fiber} $Sp_\lambda$, and the Garsia-Procesi module is well known to be isomorphic as $S_n$-modules to the cohomology ring $H^*(Sp_\lambda)$ under the star action.

\section{Descent basis for the Garsia-Procesi module}

In this section we define our proposed basis of 
$\gpring_\lambda$.

\subsection{Description of the basis}

We define the indexing set for our basis of
$\gpring_\lambda$. 
\begin{defn}
Let $\lambda=(\lambda_1,...,\lambda_l)$ be a partition of $n$, or more generally a weak composition. We define a subset
$\majtabs_\lambda\subset S_n$ by
\begin{equation}
    \label{eq:defdlambda}
\begin{split}
\majtabs_\lambda & = 
\bigcup_{\an_1,...,\an_l}
\shuff\left(\an_1,...,\an_l\right) 
=\bigcup_{(A_1|\cdots |A_l) \in \osp(\lambda)}
\left\{ \an: \an|_{A_i}\in \majtabs_{\lambda_i}\right\}.
\end{split}
\end{equation}
ranging over all $l$-tuples 
$(\an_1,...,\an_l)\in \majtabs_{\lambda_1}\times \cdots \times \majtabs_{\lambda_l}$.
\end{defn}
In other words, we shuffle together all $l$-tuples of 
descent compositions with sizes given by the parts of $\lambda$.
For instance, we would have
\[\majtabs_{(3,1)}=\{0000, 0001, 0010, 0011, 0012,    0100, 0101, 0102, 0110, 0120,
    1001,1010\}.\]
Notice that the union in \eqref{eq:defdlambda} is not disjoint, otherwise the size would always be $n!$.

In fact, we have that $\majtabs_{\lambda}$ is a subset of $\majtabs_n$:
\begin{lem}
\label{lem:shuffmaj}
Every shuffle of descent compositions is also a descent composition. In other words, $\majtabs_{\lambda}\subset \majtabs_{n}$.
\end{lem}
Thus, we can make the following definition:
\begin{defn}
Let $\lambda$ be a partition of $n$, or more generally a weak composition. We define a subset
$J_\lambda^{\maj}\subset S_n$ by
\begin{equation}
    \label{eq:defjlambdamaj}
J_\lambda^{\maj}=\left\{\majt^{-1}(\an):
\an \in \majtabs_\lambda\right\}
\end{equation}
\end{defn}

\begin{defn}
\label{def:raybas}
We define a collection of monomials by
\begin{equation}
\label{eq:raybas}
\raybas_{\lambda}=\left\{\xn^{\an}: \an\in \desc_\lambda\right\}=\left\{g_\tau(\xn): \tau \in \jmaj_{\lambda}\right\}.
\end{equation}
If we specify another weak composition $\mu$, we define
\begin{equation}
\raybas_{\lambda,\mu}=\left\{\antisym_{\mu} g_{\tau}(\xn):
\tau \in \jmaj_{\lambda} \cap \shuff'_{\mu}\right\},
\end{equation}
recalling that $\antisym_{\mu}$ is the antisymmetrizing element of $\mathbb{C}[S_{\mu}]$, and $\shuff'_{\mu}$ is the collection of reverse shuffles.
\end{defn}

For instance, we would have
\begin{equation}\label{eq:basex}
\begin{split}
\raybas_{(3,2),(2,2,1)}=&
\{x_1x_3-x_1x_4-{x_3}{x_2}+{x_2}{
x_4},\\
&{x_1}{x_3}{x_5}-{x_1}{x_4}{x_5}-{
x_2}{x_3}{x_5}+{x_2}{x_4}{x_5}, \\
& {{x_3}}^{2}{x_1}-{{x_4}}^{2}{x_1}-{{x_3}}^{2}{x_2}+{{x_4}}^{2}{x_2},\\
&{{x_3}}^{2}{x_1}{x_5}-{{x_4}}^{2}{x_1}{
x_5}-{{x_3}}^{2}{x_2}{x_5}+{{x_4}}^{2}{x_2}{x_5},\\
& {{x_5}}^{2}{x_1}{x_3}-{{x_5}}^{2}{x_1}{x_4}-{{ 
x_5}}^{2}{x_2}{x_3}+{{x_5}}^{2}{x_2}{x_4}\}.
\end{split}
\end{equation}

\begin{rem}
We could also consider a similar set $\mathcal{B}^{\inv}_\lambda$, defined as above but with $\artin_{n}$ in place of $\desc_{n}$, and $\inv^{-1}$ in place of $\majt^{-1}$. In fact, this is not a new set, and it turns out to be the the same as the basis studied in \cite{garsia1993graded}, which are characterized in terms of sub-Yamanouchi words.
\end{rem}

We may now state our first main result:

\begin{thm}
\label{thm:majbasis}
Let $\lambda$ be a partition of $n$. Then we have that $\raybas_{\lambda}$ determines a basis of $\gpring_{\lambda'}$. Moreover,
\begin{enumerate}
    \item The elements of $\raybas_{\lambda}$ 
    are the leading terms in the descent ordering,
    meaning that
    $LT_{des}(\ghideal_{\lambda'})$ is the set of all monomials not in $\raybas_{\lambda}$.
\item If $S_{\mu}\subset S_n$ is a Young subgroup, then 
$\raybas_{\lambda,\mu}$ is a basis of $\gpring_{\lambda',\mu}=\antisym_{\mu} \gpring_{\lambda'}$.
\end{enumerate}
\end{thm}

\begin{example}
We have that
\[\omega \Ht_{(2,2,1)}[X;t]=
{t}^{4}{m}_{{3,2}}+ \left( {t}^{4}+{t}^{3} \right) {m}_{{3,
1,1}}+ \left( 2{t}^{4}+2{t}^{3}+{t}^{2} \right) {m}_{{2,2,1}}
+ \]
\[\left( 3{t}^{4}+5{t}^{3}+3{t}^{2}+t \right) {m}_{{2,1,1,1
}}+ \left( 5{t}^{4}+11{t}^{3}+9{t}^{2}+4t+1 \right) {m}_{
{1,1,1,1,1}},
\]
with the coefficient of $m_{\mu}$ corresponding to 
the graded dimension of 
$\antisym_{\mu} \gpring_{(2,2,1)}$. On the other hand, we see that the coefficient $2t^4+2t^3+t^2$ of $m_{2,2,1}$ matches the graded number of basis elements  from \eqref{eq:basex}.
\end{example}

In particular, taking the coefficient of the monomial $m_{1^n}$ from $\Ht_{\lambda}[X;t]$
gives a formula for the size of $\jmaj_{\lambda}$, which turns out to be surprisingly far from obvious:
\begin{equation}
    \label{eq:jsize}
   |\jmaj_{\lambda}|=\binom{n}{\lambda'}=
    \frac{n!}{\lambda'_1! \cdots \lambda'_h!}.
\end{equation}

\subsection{Proof of linear independence}

We begin by proving linear independence, with the help of an injective map $\gpring_\lambda$ into a an explicit direct sum of tensor products of coinvariant algebras.

For any $S=\{i_1,...,i_k\}\subset \{1,...,n\}$, 
we define a map \begin{equation}
\label{eq:rtensmap}
\varphi_S : R \rightarrow \C[x_1,...,x_k]\otimes \C[x_1,...,x_{n-k}]\rightarrow R_{k} \otimes \C[x_1,...,x_{n-k}].
\end{equation}
The first map evaluates $x_{i_j}=x_j$
assuming the elements of $S$ are written in increasing order, and similarly for the complementary indices. 
The second map is the quotient map 
on the first factor.
For instance, for $n=7$, we would have
\[\varphi_{\{2,3,7\}} (x_1^2x_2x_4^3x_6)=
x_1\otimes x_1^2x_2^3x_4=(-x_2-x_3) \otimes x_1^2x_2^3x_4.\]

Then we have
\begin{lem}
\label{lem:tanimap}
Let $\lambda$ be a partition of length
$l>0$ and let
$\mu=(\lambda_2,...,\lambda_l)$
be the result of removing the first row.
For any subset $S\subset \{1,...,n\}$ of size
$k=\lambda_1$, the image of $I_{\lambda'}$
under $\varphi_{S}$ is contained in $R_k\otimes I_{\mu'}$.
\end{lem}

\begin{proof}

Suppose that $e_d(\xn_T)$ is a generator 
of $I_{\lambda'}$, where
$T\subset \{1,...,n\}$ is a subset of size $m$, 
and we are denoting
$\xn_T=\{x_{i}:i\in T\}$. 
By the description of the Tanisaki ideal, this means that
\begin{equation}
\label{eq:tanisakigen}
d>m-(\lambda_{n-m+1}+\lambda_{n-m+2}+\cdots)
\end{equation}
In fact, we will show that every term in the identity
    \begin{equation}\label{eq:tanisakisplit1}
        e_d(\xn_T) = \sum_{i} e_i(\xn_{S \cap T})e_{d- i}(\xn_{T\cap S'}),
    \end{equation}
is mapped into $R_k\otimes I_{\mu'}$ under $\varphi_S$, where $S'$ is the complement of $S$. Since the property that $e_i(\xn_{T\cap S})$ and $e_{d-i}(\xn_{T\cap S'})$ are generators in the ideals $I_k$ and $I_{\mu'}$ depends only the sizes
$l=|T\cap S|$ and $m-l=|T\cap S'|$,
it suffices to show that for any $i,l$, 
either 
$e_i(x_1,...,x_l) \in I_k$ or $e_{d-i}(x_1,...,x_{m-l})\in R_{\mu'}$.

We do this in two cases: suppose first that
$S\subset T$, so that $l=k$. Then every 
$e_i(x_1,...,x_l)$ for $i\geq 1$ is in $I_k$, so we may suppose that $i=0$.
Now if $e_{d}(x_1,...,x_{m-l})\notin I_{\mu'}$, we must have that
\[d\leq m-k-(\mu_{n-m+1}+\mu_{n-m+2}+\cdots)=\]
\[m-(\lambda_1+\lambda_{n-m+2}+\lambda_{n-m+3}+\cdots),\]
by inserting $\lambda=\mu$, $m=m-k$, $n=n-k$ into \eqref{eq:tanisakigen}. But this contradicts \eqref{eq:tanisakigen}, noting that 
$\lambda_1\geq \lambda_{n-m+1}$.

In the second case, suppose that $T$ does not contain $S$, so that $l<k$. Then if
 $e_{d-i}(x_1,...,x_{m-l})\notin I_{\mu'}$, we have
\[d-i\leq m-l-(\mu_{(n-k)-(m-l)+1}+\mu_{(n-k)-(m-l)+2}+\cdots)\leq \]
\[m-l-(\mu_{n-m}+\mu_{n-m+1}+\cdots)=\]
\[m-l-(\lambda_{n-m+1}+\lambda_{n-m+2}+\cdots).\]
Then we obtain that $i>l$, so that $e_i(x_1,...,x_l)=0$.
\end{proof}

\begin{cor}
We have a well-defined map
    \begin{equation}
\label{eq:tanimap}
\varphi_{\lambda}: R_{\lambda'}
\rightarrow \bigoplus_{(A_1|\cdots|A_l) \in 
\osp(\lambda)}
R_{\lambda_1}\otimes\cdots \otimes R_{\lambda_l}
\end{equation}
by reordering the variables according to the parts of $A_i$ as in the definition of $\varphi_S$, and reducing by $I_{\lambda_i}$ on each factor.
\end{cor}
We will see that the images of $\raybas_{\lambda}$ under $\varphi_{\lambda}$ are linearly independent. Once it is shown that they make up a basis, it will follow that $\varphi_\lambda$ is injective.

We can now prove independence:

\begin{prop}
\label{prop:independence}
Let $\an \in \majtabs_{\lambda}$, and suppose that
\begin{equation}
\label{eq:basisproof1}
\sum_{\bn \leq_{des} \an} c_{\bn} \xn^{\bn} \in I_{\lambda'}
\end{equation}
for some coefficients $c_{\bn}$. Then we must
have $c_{\an}=0$. 
\end{prop}

\begin{proof}

We prove the proposition by induction on $l$, the number of parts of $\lambda$. The base case of $l=1$ is just the statement that the descent polynomials are independent in $R_{\lambda_1}$. We therefore suppose that $l\geq 2$, and let $\mu=(\lambda_2,...,\lambda_l)$ be the result of removing the first row from $\lambda$.

Let $(A_1|\cdots |A_l)$ be an ordered set partition for which the
$\an|_{A_i}$ are descent compositions, as in the definition of $\majtabs_{\lambda}$. 
Let $S=A_1$ be the first component with
complement $T=\{1,...,n\}-S$, and let 
$\an'=\an|_S,\an''=\an|_T$ be the corresponding elements of $\desc_{\lambda_1},\desc_{n-\lambda_1}$.
We then compute the expansion 
\begin{equation}
\label{eq:linreduction}
\varphi_{S}\left(
\sum_{\bn \leq_{des} \an} c_{\bn} \xn^{\bn}\right)=
\sum_{\bn',\bn''}d_{\bn',\bn''}  \left(\xn^{\bn'} \otimes 
 \xn^{\bn''}\right)
\end{equation}
in $R_{\lambda_1} \otimes \C[x_{1},...,x_{n-\lambda_1}]$,
with respect to the descent basis on the first
factor, so that 
$\bn'\in \majtabs_{\lambda_1}$.
Then for every pair $(\bn',\bn'')$, we have that $d_{\bn',\bn''}$ is a linear combination of those coefficients $c_{\bn}$ for which
\begin{equation}
\label{eq:linereduction0}
\bn'\leq_{des} \bn|_{S},\quad
\bn''= \bn|_{T},\quad
\bn \leq_{des} \an.
\end{equation}

We would like to apply the induction hypothesis 
to the coefficient of
$\xn^{\an'}\otimes \cdots $ in the right hand
side of \eqref{eq:linreduction}, with $\an''$ in place of $\an$.
To see that \eqref{eq:basisproof1} 
holds for this coefficient, we must check that
\begin{equation}
    \label{eq:basisproof2}
\sum_{\bn''} d_{\an',\bn''} \xn^{\bn''}
\in I_{\mu'},
\end{equation}  
which follows from Lemma \ref{lem:tanimap}.
We also need to see that the coefficients 
in \eqref{eq:basisproof2} take the same triangular form as in
\eqref{eq:basisproof1}, meaning we may take $\bn''\leq_{des} \an''$.
To see this, suppose that $d_{\an',\bn''}\neq 0$ for some $\bn''$.
By 
\eqref{eq:linereduction0} with $\bn'=\an'$,
there is a composition $\bn$ with 
$\an'\leq_{des} \bn|_S$, $\bn''=\bn|_T$,
and $\bn\leq_{des} \an$.
By property \eqref{eq:desprop} of the descent order,
we must have $\bn''\leq_{des}\an''$, otherwise we would have that $\an <_{des} \bn$.

Applying the induction hypothesis, we have that 
$d_{\an',\an''}=0$. On the other hand, applying \eqref{eq:linereduction0} again with $\bn''=\an''$, 
we find that $c_{\an}=d_{\an',\an''}$, completing the proof.
\end{proof}

\begin{example}
Take a general sum $\sum_{\bn \leq_{des} \an} c_{\bn}\xn^{\bn}$ in the proof of the proposition, with $\an=011011 \in \desc_{(3,3)}$. The sum has a total of 45 terms, noting that we are including all $\bn\leq_{des} \an$, not just those terms of the same degree. Now suppose we take
\[S=\{1,3,6\},\quad T=\{2,4,5\},\quad 
\an'=011,\quad \bn'=011,\quad \an''=101,\]
in the proof of the proposition.
Extracting the coefficient of $\xn^{\bn'}\otimes \cdots$ in \eqref{eq:linreduction}, we obtain
\[d_{011,101}=c_{011011},\quad d_{011,011}=c_{00
1111},\quad d_{011,100}=c_{011001}-c_{111000},\]
\[d_{011,010}=c_{001101}-c_{101100},\quad d_{011,001}=c_{001011}-c_{101010},\quad d_{011,000}=c_{001001}-c_{101000}.\]
For each $\bn''$, the coefficients $c_{\bn}$ that appear in the expression for $d_{\bn',\bn''}$ all satisfy the properties of
\eqref{eq:linereduction0}.
Note also that $d_{\bn',\an''}=c_{\an}$,
agreeing with the last paragraph of the proof,
since $\bn'=\an'$.
\end{example}

As a consequence, we have two corollaries:
\begin{cor}\label{cor:youngind}
The set $\raybas_{\lambda,\mu}$ is linearly independent in $N_{\mu}R_{\lambda'}$.
\end{cor}
\begin{proof}
When no second partition $\mu$ is specified, this follows from the proposition. Otherwise, we notice that antisymmetrizing any monomial that comes from a reverse shuffle only adds lower terms in the descent order, so they must be independent in $N_\mu R_{\lambda'}\subset R_{\lambda'}$.
\end{proof}
\begin{cor}
\label{cor:leq}    
The size of the basis set satisfies 
\begin{equation}
\label{eq:basisleq}
    |J^{\maj}_{\lambda}|\leq \binom{n}{\lambda'}=\frac{n!}{\lambda'_1!\cdots \lambda'_h!}.
\end{equation}
\end{cor}

\subsection{Proof of the dimension count}

Recall that $\ribtabs^0_{\lambda}$ was the set of ribbon-shaped
tableaux $\tilde{T}$ satisfying two properties related to the
possible types of dinv-pair, described in Defintion \ref{def:minrib}.
Define a map $\Psi_{\lambda} : \ribtabs^0_{\lambda} \rightarrow \majtabs_\alpha$ by setting $\Psi_{\lambda}(\tilde{T})=\an$,
where $a_i$ is the height of the entry containing the number $i$.
An example is given in Example \ref{ex:alg}
below.
Since $\Psi_{n}:\ribtabs_n\rightarrow \desc_n$ 
is a bijection in the 
case of a single tableau, we have that any element
$\tilde{T}\in \ribtabs_{\lambda}$ is uniquely determined by
the data of $\Psi_{\lambda}(\tilde{T})$ together with the ordered set partition $(A_1|\cdots|A_l)$ for which $A_i$ contains the entries
that appear in $T_i$.
We will prove that an element $\tilde{T}\in \ribtabs_{\lambda}^0$ is determined by just the data of $\Psi(\tilde{T})$. In other words,
\begin{prop}
\label{prop:injective}
The restriction of $\Psi_{\lambda} : 
\rib^{0}_{\lambda}\rightarrow \desc_{\lambda}$ is 
injective.
\end{prop}
Then the reverse equality in equation \eqref{eq:basisleq}
will follow with the help of \eqref{eq:ribbonmulti}, establishing
that $\basis_{\lambda}^{\maj}$ is a basis, and that the restriction of $\Psi_{\lambda}$ is a bijection.

To prove the proposition, 
we will construct an algorithm that reproduces the ordered set partition $(A_1|\cdots|A_l)$
from the data of $\Psi_{\lambda}(\tilde{T})$,
whenever $\tilde{T} \in \ribtabs^0_{\lambda}$ is minimal. To start, we define a sequence
$(\tilde{a}_1,\tilde{a}_2,...)$ out of $\an$
by setting $a_{kn+i}=a_i+k$ for all $k\geq 0$, and $i\in \{1,...,n\}$, such as the terms in Example \ref{ex:alg} below. The algorithm is described:

\begin{algo}\label{alg:ray}
Starting with a descent composition $\an$, and a partition $\lambda=(\lambda_1,...,\lambda_l)$, we will produce an ordered set partition
$(A_1|\cdots|A_l)\in \osp(\lambda)$.
\end{algo}
\begin{enumerate}
    \item For each $k$ from $1$ to $l$, do the following:
    \begin{enumerate}
        \item Initialize $i=0$, and $A_k=\emptyset$.
    \item \label{item:algnextind} For $m$ from 0 to $\lambda_k-1$,
locate the smallest index 
$i'=an+j>i$ for which
$\tilde{a}_{i'}=m$, and $j$ add not yet been 
added to any $A_1,..,A_k$. If none exists, we say the algorithm does not terminate.
\item Set $A_k=A_k \cup \{j\}$, and 
set $i=i'$.
    \end{enumerate}   
\item Return $(A_1|\cdots|A_l)$.
\end{enumerate}
If $\an=\Psi_{\lambda}(\tilde{T})$, it
is useful to describe the result of Algorithm 
\ref{alg:ray} directly in terms of $\tilde{T}$.
It is easy to see that the following is equivalent:
\begin{algoa}\label{alg:ray1}
Let $\tilde{T}\in \ribtabs_{\lambda}$. Recall that $\sigma(x)$ denotes the value of the entry in box $x$.
\end{algoa}
\begin{enumerate}
    \item For each $k$ from $1$ to $l$, do the following:
    \begin{enumerate}
        \item Let $x$ be the square in
    the bottom row of $\tilde{T}$ containing
the smallest available
(i.e. has not yet been added to any $A_1,...,A_k$) value of $\sigma(x)$. Initialize $A_k=\{\sigma(x)\}$.
\item \label{item:algnextind1} For $m$ from 1 to $\lambda_k-1$, do the following:
\begin{enumerate}
    \item If there is an available square $x'$ in the row above $x$ with $\sigma(x')>\sigma(x)$, let $x'$ be the one for which that value is the lowest.
    \item Otherwise, let $x'$ be the available square with the lowest value of 
    $\sigma(x')$ in the highest row which is not higher than that of $x$.
\end{enumerate}
If either case is successful, 
set $A_k=A_k \cup \{x'\}$, and $x=x'$.
Otherwise, the algorithm does not terminate.
    \end{enumerate}   
\item Return $(A_1|\cdots|A_l)$.
\end{enumerate}
In order to prove Proposition \ref{prop:injective},
it suffices to show that the sets $A_k$ resulting from Algorithm 1$'$ contain precisely the entries of $T_k$
whenever 
$\tilde{T}\in \ribtabs_{\lambda}^0$
is minimal.

\begin{example}
\label{ex:alg}
Consider the element $\tilde{T}\in 
\ribtabs^0_{6,2,1}$, and its corresponding descent composition given by
$$ \Psi_{6,2,1} :
       \begin{ytableau}
        2 & 4 & \none & \none & \none & \none & \none & \none \\
        \none & 1 & 6 & 7 & \none & 9 & \none & \none  \\
        \none & \none & \none & 3 & \none & 5
        & \none & 8
    \end{ytableau}
    \mapsto (1,2,0,2,0,1,1,0,1).
    $$
Then the sequence 
    $(\tilde{a}_1,\tilde{a}_2,...)$ 
    is given by
    \[(1,2,0,2,0,1,1,0,1|2,3,1,3,1,2,2,1,2|3,4,2,4,...)\]
Algorithm \ref{alg:ray} returns the set partition
$(361246|59|8)$, where we are writing the elements in order to describe the order in which they are added. We can see that this is the same sequence given by Algorithm 1$'$ applied to $\tilde{T}$, and that the elements of $A_k$ are the entries of $T_k$.    
\end{example}

Before proving the proposition, we need two lemmas about the minimal tableaux in
$\ribtabs_{\lambda}$.
\begin{lem}
\label{lem:patterns}
The elements $\tilde{T} \in \ribtabs^0_{\lambda}$ satisfy the following rules. Each pattern below is assumed to be in separate components $T_i,T_j$
with $i<j$, as there are no dinv pairs within a single one.
\begin{enumerate}
\item \label{item:pattern1} The entries in the lowest row
are in decreasing order.
    \item  \label{item:pattern2} Whenever we have the pattern $$ 
    \begin{ytableau}
        a & \none & \none & \none & \none & \none & b \\
        c & \none & \none & \none&\none & \none & \none
    \end{ytableau}
    $$
 we cannot have $c<b<a$.
\item \label{item:pattern3}  If we have the pattern
$$ 
    \begin{ytableau}
        a & b & \none & \none & \none & \none & c\\
    \end{ytableau}
    $$
  we must have
    $a<c$.
\item \label{item:pattern4}  The following pattern never appears:
$$ 
    \begin{ytableau}
    \none & \none & \none & \none & \none & \none & e\\
        a & b & c & \none & \none  & \none  & d\\
    \end{ytableau}
    $$
\item \label{item:pattern5} If we have
$$ 
    \begin{ytableau}
    a & \none & \none & \none & \none & \none & d\\
        b & c & \none & \none & \none & \none & e\\
    \end{ytableau}
    $$
    then we must have $a<c$.
    \end{enumerate}
\end{lem}
\begin{proof}
In every case, we can show that a forbidden pattern contradicts one of the criteria
in Definition \ref{def:minrib}, ether that there is a dinv pair in the first row, or that one of the squares on the right is contained in two-dinv pairs with the component on the left.
We prove the final item. First, we must
have that $b<d<c$, otherwise there would be two dinv pairs containing either $d$ or $e$, since
$e<d$. But we also have $a<d$, otherwise $e$ would be in two dinv pairs.
\end{proof}

\begin{figure}
    \centering
    \includegraphics{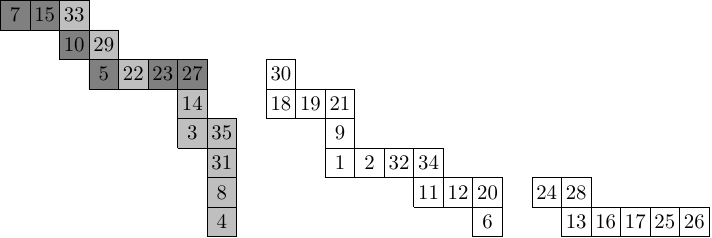}
    \caption{Two phases of the Algorithm
    1$'$ applied to the first component of a ribbon
    tableau $\tilde{T} \in \ribtabs^0_{15,13,7}$.
    Light gray corresponds to the ascending phase, and dark gray is the descending phase.}
    \label{fig:twophase}
\end{figure}

We can now prove Proposition \ref{prop:injective}.
\begin{proof}
We pointed out above that a ribbon tableau
$\tilde{T}$ is determined uniquely by 
the data of $\Psi_{\lambda}(\tilde{T})$
together with the set of entries $\sigma(x)$
in each component $T_i$.
    It therefore suffices to show that the ordered set partition $(A_1|\cdots|A_l)$ resulting from Algorithm 1$'$ terminates, and 
    that $A_i$ contains exactly the entries of $T_i$ for each $i$.
    We may prove this just for the first component $T_1$, as the remaining parts would follow by induction.

    We show that the algorithm proceeds into two different phases, ascending and descending, shown in Figure \ref{fig:twophase}. First, we claim that
    Algorithm 1$'$ will reach the top row of $T_1$, and that each square (shown in gray in the figure) added during that time is contained in $T_1$.
By item \ref{item:pattern1} of Lemma \ref{lem:patterns}, the first box will always be in $T_1$. In items \ref{item:pattern2} and \ref{item:pattern3}, we will never select the square on the right in the corresponding picture, and hence will stay in $T_1$.
Since there is always a path to the top in within a single tableau, we will reach the top.

The descending phase consists of all those remaining steps following the ascending phase,
shown in dark gray.
We claim that all boxes remaining in $T_1$
must reside in the row of the highest boxes
among $(T_2,...,T_l)$, or higher. In other words, the dark gray region of the figure, 
interpreted as the complement of the light gray region within $T_1$,
is entirely above or equal to the maximum height all those squares in the subsequent figures, which in the example would be row 5 containing the box $\sigma(x)=30$.
The reason is that any row which contains only a single square of $T_1$ must have been crossed during the ascending phase. But by item
\ref{item:pattern4},
any row below the maximum height must contain at most two boxes in $T_1$. By item 
\ref{item:pattern5}, both boxes in such a row would have to have been crossed already (we can only enter through $c$, and we would have to go through $b$ to get to $a$ in the corresponding pattern).

Thus, the second phase of the algorithm must simply choose each entry in order from $T_1$ until the final row which is the same height as the highest box from the remaining tableaux.
In this row, we must still choose all remaining entries from within $T_1$, otherwise we would have necessarily end up in the situation of item
\ref{item:pattern2} again.
\end{proof}

\begin{proof}[Proof of Theorem \ref{thm:majbasis}]
We have seen that $\Phi_{\lambda}: \ribtabs_{\lambda}^0\rightarrow \desc_{\lambda}$ is a injective, so 
by Corollary \ref{cor:leq}, it must be a bijection. Since the corresponding permutation $\sigma\in \jlamaj$ corresponds to the reading word of $\tilde{T}$, we find that
$|\jlamaj\cap \shuff(\mu)|$
and 
$|\jlamaj \cap \shuff'(\mu)|$ are the corresponding coefficients of the Hall-Littlewood polynomials from Proposition
\ref{prop:ribbonhall}.
Since the bases are linearly independent
by Corollary \ref{cor:youngind} and have the same size, they must be a basis. The form of Proposition \ref{prop:independence} shows that the leading term statement.
\end{proof}

\begin{cor}
\label{thm:hl}
The modified Hall-Littlewood polynomial 
is given by 
$\Ht_\lambda[X;0,t]=\Ct_\lambda[X;t]$ where
\begin{equation}
    \begin{split}\Ct_\lambda[X;t] & = \sum_{\mu \vdash n}\Big( \sum_{\sigma \in \jmaj_{\lambda} \cap \shuff(\mu)} t^{\maj(\sigma)}\Big)m_\mu,\\
\omega \Ct_\lambda[X;t] & = \sum_{\mu \vdash n}\Big( \sum_{\sigma \in \jmaj_{\lambda} \cap \shuff'(\mu)} t^{\maj(\sigma)}\Big)m_\mu.
\end{split}
\end{equation}
\end{cor}
We also obtain an alternate expression for $R_{\lambda}$:
\begin{cor}
The map $\varphi_\lambda$
from \eqref{eq:tanimap} is injective, 
and its image is isomorphic to $R_{\lambda}$.
\end{cor}

\bibliographystyle{plain}
\bibliography{refs}

\end{document}